\documentclass[11pt]{amsart}

\usepackage[T1]{fontenc}
\usepackage[utf8]{inputenc}
\usepackage{lmodern}
\usepackage{microtype}
\usepackage{amsmath,amssymb,amsthm,mathtools}
\usepackage{aliascnt}
\usepackage{enumitem}
\usepackage{geometry}
\usepackage{xcolor}
\usepackage{tikz}
\usepackage{hyperref}
\usepackage{comment}
\usepackage[nameinlink,noabbrev]{cleveref}
\usepackage{lineno}

\hypersetup{
  colorlinks=true,
  linkcolor=blue!45!black,
  citecolor=blue!45!black,
  urlcolor=blue!45!black,
  pdftitle={Short Homology Bases For Translation Surfaces},
}

\newtheorem{theorem}{Theorem}[section]

\newaliascnt{proposition}{theorem}
\newtheorem{proposition}[proposition]{Proposition}
\aliascntresetthe{proposition}

\newaliascnt{lemma}{theorem}
\newtheorem{lemma}[lemma]{Lemma}
\aliascntresetthe{lemma}

\newaliascnt{corollary}{theorem}
\newtheorem{corollary}[corollary]{Corollary}
\aliascntresetthe{corollary}

\theoremstyle{definition}
\newaliascnt{definition}{theorem}
\newtheorem{definition}[definition]{Definition}
\aliascntresetthe{definition}

\theoremstyle{remark}
\newaliascnt{remark}{theorem}
\newtheorem{remark}[remark]{Remark}
\aliascntresetthe{remark}

\newaliascnt{question}{theorem}

\aliascntresetthe{question}

\newaliascnt{conjecture}{theorem}

\aliascntresetthe{conjecture}

\newcommand{\Area}{\operatorname{area}}
\newcommand{\dist}{\operatorname{dist}}
\newcommand{\im}{\operatorname{im}}
\newcommand{\inter}{\operatorname{Int}}
\newcommand{\rank}{\operatorname{rank}}
\newcommand{\sys}{\operatorname{sys}}

\newcommand{\N}{\mathbb{N}}

\newcommand{\Z}{\mathbb{Z}}
\newcommand{\R}{\mathbb{R}}
\newcommand{\cH}{\mathcal{H}}
\newcommand{\Gmin}{G_{\min}}

\begin{document}

\title[Short Homology Bases for Translation Surfaces]{Short homology bases for translation surfaces}
\author{Peter Buser, Achintya Dey, Eran Makover and Bjoern Muetzel}
\date{\today}

\begin{abstract}
Let $S$ be a closed translation surface of genus $g\ge 2$ with $\Area(S)= 4\pi g$. We show that for any $\lambda \in (0, 1)$ there exist $\lfloor \lambda \cdot g\rfloor$ homologically independent simple closed curves of length at most $ C(\lambda) \cdot \log(g)$, where $C(\lambda)$ is a constant that depends only on $\lambda$. This result is obtained from a mixed graph construction based on a Voronoi graph of the surface with the cone points as seeds and its dual graph. We also give a complementary result using only the Voronoi graph that produces $2g$ short loops that form a homology basis. This construction, however, depends on the length of a shortest saddle connection of the surface $S$.
\end{abstract}

\maketitle

\section{Introduction}

A \textit{translation surface} $S$ is a connected surface obtained by identifying pairwise by translations the sides of a set of plane polygons. It follows that $S$ is a closed flat surface with singularities or cone points with cone angles $2\pi \cdot (k+1)$ for $k \in \N$ (see \cite{ma},\cite{wr} and \cite{zo} for an overview). As the coordinate changes are translations and therefore holomorphic functions, these surfaces are naturally Riemann surfaces. A saddle connection is a geodesic segment joining cone points, possibly the same cone point, with no cone point in its interior. Important geometric invariants on these surfaces are the length of a shortest saddle connection, the systole, which is the length of a shortest noncontractible loop, and the length spectrum of a minimal homology basis (see \cite{bmm1} for details). The latter provides a bridge between the metric geometry of a surface and the geometry of its Jacobian variety \cite{bps,bs,bmm2, mu}. Via the Jacobian, from the bounds on these curves, a rough fundamental domain of the moduli space can be obtained (see Introduction  of \cite{bmm2}). 
\\
Using different graph or mesh approaches, families of short homologically independent loops on general, hyperbolic and translation surfaces have been obtained in \cite{bps,bmm2,bmm3, ka}. Our work was inspired by Erickson and Whittlesey's use of the cut locus to find short homotopy bases on oriented combinatorial 2-manifolds in \cite{ew}. Here we consider the natural extension of the cut locus of a point, the Voronoi graph or diagram. The carriers for our constructions are the Voronoi graph whose seeds are the cone points of the translation surface and its dual graph. \\
\\
The first main construction, and a central new ingredient of this paper, is the \textit{mixed shorter-edge graph}. Let $G_{VD}$ be the graph obtained by combining the Voronoi graph with its dual graph. For each dual pair of edges $(e,e^*)$ from $G_{VD}$ we keep whichever of $e$ and $e^*$ is shorter. The resulting graph $G_{\min}$ both generates a primitive subgroup of the integral homology and is of controlled total length. This geometric and topological control allows us to obtain the subsequent estimates for short integral homology generators. We recall that a subgroup $U \subseteq H_1(S,\Z)$ is \textit{primitive} if $H_1(S,\Z) / U$ is torsion-free, equivalently, any $\Z$-basis of $U$ can be extended to one of $H_1(S,\Z)$.

\begin{lemma}[Mixed shorter-edge graph] Let $S$ be a translation surface of genus $g \geq 2$. Let $G_{\min}$ be the shorter-edge graph of the Voronoi graph $G_V$ combined with its dual. Then for the total length $L(G_{\min})$ of the edges of this graph we have
\[
 L(G_{\min})\le \sqrt{2N\,\Area(S)},
\]
where $N$ is the number of edges of $G_V$. Furthermore the image of $H_1(G_{\min},\Z)$ in $H_1(S,\Z)$ is a primitive subgroup of rank at least $g$. 
\label{lem:Gmin_intro}
\end{lemma}
Notably this result is independent of the length $\ell(\delta)$ of a shortest saddle connection $\delta$. The result on the rank of the first homology image is established by constructing an incident graph for the connected components of a regular neighborhoood of $G_{min}$ in $S$. We then establish the length estimates using another result from graph theory. The main point is to keep track of the number of curves that can be extracted. After passing to a minimal homology core of rank at least $g$, the Bollobas-Szemeredi-Thomason graph inequality can be applied iteratively. This gives the following explicit estimate:

\begin{theorem}[Half rank log-bound] Let $S$ be a translation surface of genus $g \geq 2$ and $\Area(S)=4\pi g$. Then there exist $g$ homologically independent loops $(\alpha_i)_{i=1,\ldots,g}$ that can be extended to a homology basis of $H_1(S,\Z)$, such that 
\[ 
 \ell(\alpha_i)
 \le 
 \frac{70 g}{g-i+1} \cdot  \log(g-i+2) \qquad \text{ \ for all \ } i \in \{1,2,\ldots,g \}.
\]
\label{thm:half_rank_intro}
\end{theorem}

Expressing this result in terms of a portion of the homology, we have:

\begin{corollary} \label{cor:prop_intro}
Let $S$ be a translation surface of genus $g \geq 2$ and $\Area(S)=4\pi g$.  Then for any $\lambda \in (0,1)$ there are $\lfloor \lambda \cdot g \rfloor $ homologically independent loops $(\alpha_i)_{i=1,\ldots,\lfloor \lambda \cdot g \rfloor}$ that can be extended to a homology basis of $H_1(S,Z)$, such that
\[
 \ell(\alpha_i)
 \le
 \frac{70}{1-\lambda}\log(g+1)    \qquad \text{ \ for all \ } i \in \{ 1,2,\ldots,\lfloor \lambda \cdot g \rfloor \}. 
\]
\end{corollary}

This is the first main result of the paper.
\\
A \textit{cut graph} $G_f$ of a closed surface $S$ is an embedded graph, such that $S \backslash G_f$ is a single topological disk. The second construction starts from the full Voronoi graph $G_V$ and produces a cut graph. A standard tree-cotree procedure deletes enough Voronoi edges to merge all complementary Voronoi cells into one disk. The remaining cut graph $G_f$ has first Betti number $2g$, and induces an isomorphism on the first homology. In this case we obtain: 

\begin{theorem}(Short integral homology basis) Let $S$ be a translation surface of genus $g \geq 2$ with $\Area(S)=4\pi g$. Let $\delta$ be a shortest saddle connection. Let $G_V$ be the Voronoi graph of $S$ with the cone points as seeds. Let $G_f$ be a cut graph of $G_V$ with first Betti number $2g$. Then for the total length of the edges of the graph we have:
\[
 L(G_f)\le L(G_V)
 \le \frac{2\Area(S)}{\ell(\delta)}.
\]
Furthermore there exist $2g$ loops $(\alpha_i)_{i=1,\ldots,2g}$ whose homology classes form a basis of $H_1(S,\Z)$, such that
\[
 \ell(\alpha_i)
 \le \frac{32 \pi g}{(2g -i +1) \cdot \ell(\delta)}\log(2g-i+2)  \qquad \text{ \ for all \ } i \in \{1,2,\ldots,2g \}.
\]
\label{thm:full_rank_intro}
\end{theorem}

Thus the mixed graph and the Voronoi subgraph give complementary information: the first is uniform with respect to degeneration of the shortest saddle connection, while the second extends to the full rank. 

The paper is organized as follows. In \Cref{sec:geometry} we gather some information about the Voronoi graph and its dual, establish the length estimates for the mixed graph and provide the tree-cotree cut graph construction for the Voronoi graph.  In \Cref{sec:duality} we prove that the graph $\Gmin$ contains at least $g$ linearly independent loops of $H_1(S,\Z)$ that can be extended to a basis. In \Cref{sec:graph} we formulate the metric-graph systolic argument for an arbitrary number of cycles and prove the length estimates from Theorems \ref{thm:half_rank_intro} and \ref{thm:full_rank_intro}.

\section{The Voronoi graph and its dual}\label{sec:geometry}

A metric graph $G$ is a finite graph whose edges have positive length. Its \textit{systole} $\sys(G)$ or \textit{girth} is the length of a shortest noncontractible simple cycle. We recall that the (total) \textit{length} $L(G)$ of a graph is the sum of the lengths of all edges of $G$. Let $G$ be a graph embedded in a surface $S$. The graph $G$ is \textit{cellular} if the faces of $G$ are topological disks. If $G$ is a cellular graph, then so is its dual graph $G^*$. We start by gathering facts about the Voronoi graph whose seeds are the cone points of a translation surface and its dual. Both graphs are cellular graphs.\\

Let $S$ be a closed translation surface of genus $g\ge2$ with $n$ cone points $(p_i)_{i=1,\ldots,n}$, where the cone angle at $p_i$ is $2\pi(k_i +1)$. Then $S$ is in the stratum $\mathcal{H}(k_1,k_2,\ldots,k_n)$ and it follows from the Euler characteristic of the surface that 
\begin{equation}
      \sum_{i=1}^n k_i = 2g-2.
\label{eq:Euler}
\end{equation}

The \textit{cut locus} $CL(X)$  of a subset $X \subset S$ is the set of points for which there are at least two shortest paths to $X$. More precisely, let $\gamma_{pq}$ be a shortest geodesic arc between two points $p$ and $q$ on $S$. Then
\begin{equation*}
   CL(X) := \{p \in S \mid \exists \gamma_{x,p},\gamma_{x',p}, \gamma_{x,p} \neq \gamma_{x',p}, \text{ with } x,x' \in X \text{ and } \dist(X,p)=\ell(\gamma_{x,p}) = \ell(\gamma_{x',p}) \}.
\label{eq:cut_locus}
\end{equation*}
The \textit{Voronoi graph} $G_V$ with seeds $(p_i)_{i=1,\ldots,n}$ is the cut locus of these points in $S$.
\[
              G_V = CL(\{ p_1,p_2,\ldots,p_n \})
\]
As the surface is Euclidean outside the cone points, the edges of the Voronoi graph are straight line segments. It follows from the definition that at least three edges meet at each vertex of $G_V$. Cutting the surface along $G_V$ we obtain the $n$ \textit{Voronoi regions} $(R_i)_{i=1,\ldots,n}$, where $p_i \in R_i$. The Voronoi regions are topological disks that tessellate the surface $S$. Let $G^*_V = G_D$ be the metric dual graph of $G_V$, such that the vertices of $G_D$ are the cone points and the edges are straight line segments that intersect the edges of $G_V$ perpendicularly. We note that in the case of a generic surface $S$ the graph $G_D$ is a Delaunay triangulation of the surface. For each edge $e$ of $G_V$, let $e^*$ be the dual edge in $G_D$ that intersects $e$. For example, Figure~\ref{fig:square_tiled_duality} depicts the Voronoi graph and its metric dual graph for an L-shaped square-tiled translation surface.
 Set
\[
 a_e=\ell(e),\qquad b_e=\ell(e^*).
\]
Denote by $G_{VD}$ the graph we obtain by combining the Voronoi graph and its dual graph. Let $E: = E(G_V) $ and $E(G_D)$ be the set of edges of $G_V$ and $G_D$ respectively. Then
\begin{equation}
 N= |E| = |E(G_V)|=|E(G_D)|
\label{eq:number_edges_N}
\end{equation}
is also the number of dual edge pairs.

\begin{figure}[htbp]
    \centering
    \begin{tikzpicture}[scale=2.5]
        
        \coordinate (V00) at (0, 0);
        \coordinate (V10) at (1, 0);
        \coordinate (V20) at (2, 0);
        \coordinate (V21) at (2, 1);
        \coordinate (V11) at (1, 1);
        \coordinate (V12) at (1, 2);
        \coordinate (V02) at (0, 2);
        \coordinate (V01) at (0, 1);
        
        \coordinate (C1)  at (0.5, 0.5); 
        \coordinate (C2)  at (1.5, 0.5); 
        \coordinate (C3)  at (0.5, 1.5); 

        \draw[thick, blue!70!black] (C1) -- (C2) node[pos=0.25, below] {\tiny$e_1$};
        \draw[thick, blue!70!black] (C1) -- (C3) node[pos=0.25, left] {\tiny$e_2$};
        
        \draw[thick, blue!70!black, dashed] (C1) -- (0, 0.5) node[midway, below] {\tiny$e_3$};
        \draw[thick, blue!70!black, dashed] (C2) -- (2, 0.5) node[midway, below] {\tiny$e_3$};
        
        \draw[thick, blue!70!black, dashed] (C1) -- (0.5, 0) node[midway, right] {\tiny$e_4$};
        \draw[thick, blue!70!black, dashed] (C3) -- (0.5, 2) node[midway, left] {\tiny$e_4$};

        \draw[thick, blue!70!black, dashed] (C2) -- (1.5, 0) node[midway, right] {\tiny$e_5$};
        \draw[thick, blue!70!black, dashed] (C2) -- (1.5, 1) node[midway, right] {\tiny$e_5$};

        \draw[thick, blue!70!black, dashed] (C3) -- (0, 1.5) node[midway, above] {\tiny$e_6$};
        \draw[thick, blue!70!black, dashed] (C3) -- (1, 1.5) node[midway, above] {\tiny$e_6$};

        \draw[thick, black] (V10) -- (V11) node[pos=0.75, right] {\tiny$e_1^*$};
        \draw[thick, black] (V01) -- (V11) node[pos=0.75, above] {\tiny$e_2^*$};

        \draw[thick, black] (V00) -- (V10) node[midway, below] {\tiny$e_4^*$};
        \draw[thick, black] (V10) -- (V20) node[midway, below] {\tiny$e_5^*$};
        \draw[thick, black] (V20) -- (V21) node[midway, right] {\tiny$e_3^*$};
        \draw[thick, black] (V21) -- (V11) node[midway, above] {\tiny$e_5^*$};
        \draw[thick, black] (V11) -- (V12) node[midway, right] {\tiny$e_6^*$};
        \draw[thick, black] (V12) -- (V02) node[midway, above] {\tiny$e_4^*$};
        \draw[thick, black] (V02) -- (V01) node[midway, left]  {\tiny$e_6^*$};
        \draw[thick, black] (V01) -- (V00) node[midway, left]  {\tiny$e_3^*$};

        \foreach \point in {V00, V10, V20, V21, V11, V12, V02, V01}
            \fill[red!80!black] (\point) circle (0.8pt);

        \foreach \point in {C1, C2, C3}
            \fill[blue!70!black] (\point) circle (0.8pt);

        \node[black, above right, xshift=0.5pt, yshift=0.5pt] at (C1) {\tiny$v_1$};
        \node[black, below left, xshift=-0.5pt, yshift=0.5pt] at (C2) {\tiny$v_2$};
        \node[black, below left, xshift=0.5pt, yshift=0.5pt] at (C3) {\tiny$v_3$};

        \node[red!80!black, below left] at (V00) {\tiny$p$};
        \node[red!80!black, below left] at (V10) {\tiny$p$};
        \node[red!80!black, below left] at (V01) {\tiny$p$};
        \node[red!80!black, below left] at (V11) {\tiny$p$};
        \node[red!80!black, above right] at (V12) {\tiny$p$};
        \node[red!80!black, below right] at (V20) {\tiny$p$};
        \node[red!80!black, above right] at (V21) {\tiny$p$};
        \node[red!80!black, above left] at (V02) {\tiny$p$};

    \end{tikzpicture}
    \caption{The blue network represents the Voronoi graph $G_V = CL(\{p\})$ with vertices $V(G_V) = \{v_1, v_2, v_3\}$ and edges $E(G_V) = \{e_1, \dots, e_6\}$. The solid black lines represent the metric dual graph $G_D$ with the single vertex  $ p$ and edges $E(G_D) = \{e_1^*, \dots, e_6^*\}$. Each edge $e_i$ perpendicularly bisects its corresponding dual edge $e_i^*$.}
    \label{fig:square_tiled_duality}
\end{figure}
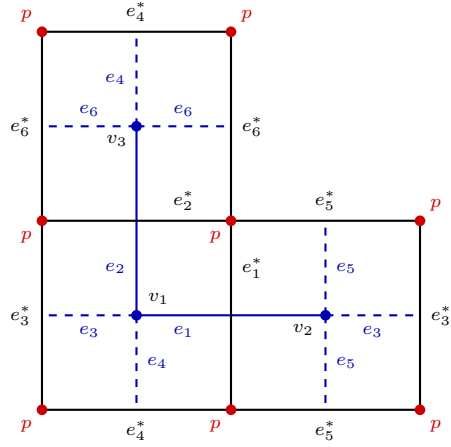

\begin{definition}[Mixed shorter-edge graph]
For each dual pair $(e,e^*)$, choose $e$ if $a_e\le b_e$ and choose $e^*$ if $b_e<a_e$. Let $\Gmin$ be the subgraph of $G_{VD}$ formed by the union of the selected edges and the attached vertices. We call $\Gmin$ the mixed shorter-edge graph.
\end{definition}

The graph $\Gmin$ need not be connected. Indeed, a chosen edge $e$ of $\Gmin$ that is an edge of the Voronoi graph never crosses an edge $d$ of $\Gmin$ that is in the dual graph $G_D$, because each edge $e$ of $G_V$ crosses only its dual edge. So the edges of $G_V$ in $\Gmin$ are disconnected from the edges of $G_D$ in $\Gmin$. The appropriate topological invariant is therefore the rank of the image
\[
 H_1(\Gmin,\Z)\longrightarrow H_1(S,\Z).
\]

\begin{proposition}[Area identity]\label{prop:area}
With the notation above,
\[
 \Area(S)=\frac{1}{2} \sum_{e\in E}a_e b_e.
\]
\end{proposition}

\begin{proof}Let $S$ be a translation surface with $n$ cone points. Let $R_i$ be a Voronoi region around the cone point $p_i$. The vertices of the region $R_i$ are the vertices of $G_V$ around the point $p_i$. Connecting the vertices with the cone point we obtain a triangle $\Delta_e$ for each edge $e \subset \partial R_i$ on the boundary of $R_i$ (see Figure~\ref{fig:Voronoi region}). As $e^*$ intersects $e$ at an angle of 90 degrees and as the height of $\Delta_e$ is $\frac{\ell(e^*)}{2}$, the area of the triangle is 
\[
      \Area(\Delta_e) = \frac{1}{4} \ell(e)\cdot \ell(e^*)
\] 

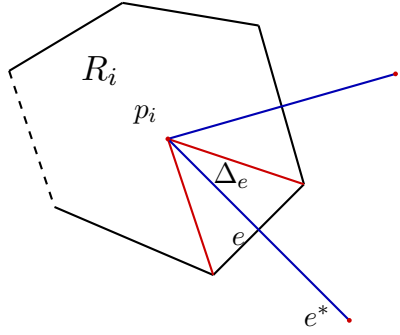
\begin{figure}[htbp]
    \centering
    \begin{tikzpicture}[scale=0.6]
        \coordinate (V1) at (3, -1);      
        \coordinate (V2) at (1, -3);      
        \coordinate (V3) at (-2.5, -1.5); 
        \coordinate (V4) at (-3.5, 1.5);  
        \coordinate (V5) at (-1, 3);      
        \coordinate (V6) at (2, 2.5);     

        \coordinate (Pi) at (0, 0);

        \draw[thick] (V3) -- (V2);
        \draw[thick] (V2) -- (V1); 
        \node at (1.56,-2.2) {$e$};
        \draw[thick] (V1) -- (V6);
        \draw[thick] (V6) -- (V5);
        \draw[thick] (V5) -- (V4);

        \draw[thick, dashed] (V4) -- (V3);

        \draw[thick, red!80!black] (Pi) -- (V1);
        \draw[thick, red!80!black] (Pi) -- (V2);

        \coordinate (EstarEnd) at (4, -4); 
        
        \coordinate (BlueEnd) at (5.019, 1.434); 
        
        \draw[thick, blue!70!black] (Pi) -- (EstarEnd);
        \draw[thick, blue!70!black] (Pi) -- (BlueEnd);

        \fill[red!80!black] (Pi) circle (1.5pt);
        \fill[red!80!black] (EstarEnd) circle (1.5pt);
        \fill[red!80!black] (BlueEnd) circle (1.5pt);

        \node[black, above left, yshift=2pt] at (Pi) {$p_i$};
        \node[black, scale=1.2] at (-1.5, 1.5) {$R_i$};
        \node[black] at (1.4, -0.8) {$\Delta_e$};
        
        \node[black, left] at (3.8, -3.9) {$e^*$};

    \end{tikzpicture}
    \caption{Voronoi region $R_i$ around the cone point $p_i$ and triangle $\Delta_e$ for the edge $e \subset \partial R_i$.}
    \label{fig:Voronoi region}
\end{figure}

As $S$ is tessellated by the Voronoi regions we have
\[
    \Area(S) = \sum_{i=1}^n \Area(R_i) = \sum_{i=1}^n \sum_{e \subset \partial R_i}   \Area(\Delta_e) = \frac{1}{2} \sum_{e \in E}  \ell(e)\cdot \ell(e^*).
\]
Here the last equation is true as each edge $e$ is used twice in the decomposition.
\end{proof}

The ideas for this proof can also be found in \cite{bg}. Since exactly one edge is selected from each pair,
\[
 L(\Gmin)=\sum_{e\in E(G_V)}\min(a_e,b_e).
\]
Using standard inequalities we obtain:

\begin{proposition}[Total length of graphs]\label{prop:mixedlength}
Let $G_V$ be the Voronoi graph of $S$ and $\Gmin$ be the mixed shorter edge graph. Let $\ell(\delta)$ be the length of a shortest saddle connection $\delta$. Then
\begin{itemize}
\item[1.)] $\displaystyle L(G_V) \leq \frac{2 \Area(S)}{\ell(\delta)}$.
\item[2.)]  $L(\Gmin)\le \sqrt{2N\,\Area(S)}$.
\end{itemize}
\end{proposition}

\begin{proof} 1.) As all edges $e^*$ of $G_D$ are saddle connections we have that $\ell(e^*) \geq \ell(\delta)$, for a shortest saddle connection $\delta$. So 
\[
   2 \Area(S) = \sum_{e \in E}  \ell(e)\cdot \ell(e^*) \geq   \sum_{e \in E}  \ell(e)\cdot \ell(\delta) \geq L(G_V)\cdot \ell(\delta).
\] 
Therefore we obtain 1.) by dividing both sides by $\ell(\delta)$.\\

2.) We have that $\min(a_e,b_e) \leq \sqrt{a_e \cdot b_e}$. So by \Cref{prop:area} and using the Cauchy-Schwarz inequality we get:
\[
 L(\Gmin) = \sum_{e \in E} \min(a_e,b_e)
 \le\sum_{e \in E}\sqrt{a_eb_e}
 \le\sqrt{N\sum_{e \in E} a_e b_e}
 =\sqrt{2N\,\Area(S)}.
\]
\end{proof}

\begin{proposition}[Edge count]\label{prop:edgecount}
Let $G_V$ be the Voronoi graph of translation surface $S$ whose seeds are the $n$ cone points of $S$. Then
\[
 N \le 3(2g-2+n) \leq 12(g-1).
\]
\end{proposition}

\begin{proof}
Let $v$ be the number of vertices of $G_V$. As $n$ is the number of faces of $G_V$, Euler's formula gives
\[
 v-N+n=2-2g .  
\]
Since every vertex has degree at least three and each edge belongs to two vertices, $3v\le 2N$.  Substituting $v=2-2g-n+N$ yields the first inequality.

As the vertices of the dual graph $G_D$ are the cone points $(p_i)_i$, we have that $k_i\ge 1$ and by \eqref{eq:Euler} 
\[
\sum_{i=1}^n k_i=2g-2,
\]
so $n\le2g-2$. This yields the second inequality.  
\end{proof}
Hence $N=O(g)$ and, under the normalization $\Area(S)=4\pi g$, the mixed graph has total length $O(g)$. Equality in $N \le 3(2g-2+n)$ holds iff every face of $G_D$ is a triangle. So we have equality in the generic case of a triangulation.\\

We now describe a graph construction that will also clarify the topology behind the later homology arguments. Let $S$ be a translation surface with $n$ cone points. When $n>1$, one can delete a controlled family of edges to merge the faces of the Voronoi graph to obtain a subgraph that generates the fundamental group $\pi_1(S)$. This tree-cotree decomposition can be found in \cite{ep}. We sketch the decomposition here for completeness. 

\begin{lemma}[Tree-cotree lemma]\label{lem:cut graph}
Let $G$ be a cellular graph on a closed orientable surface $S$ of genus $g \geq 1$. Then there exists a connected subgraph $G_f \subseteq G$ such that $S\setminus G_f$ is an open disk. Moreover,
\[
 b_1(G_f)=2g,
 \qquad
 L(G_f)\le L(G) \text{ \ and \ } 
\pi_1(G_f)\twoheadrightarrow  \pi_1(S).
\]
\end{lemma}

\begin{proof}
Let $V,E,F$ denote the numbers of vertices, edges, and faces of the primal graph $G$ in $S$.  Choose a spanning tree $T\subset G$.  It has $V-1$ edges. Consider a dual metric graph $G^*$ in $S$.
Choose a spanning tree $T^*$ of $G^*$ that does not intersect $T$. It has $F-1$ edges.  Every edge of $T^*$ crosses a unique primal edge not belonging to $T$. Delete from $G$ precisely those $F-1$ crossed edges. Let $G_f$ denote the resulting subgraph.

The graph $G_f$ remains connected because it contains $T$.  Each deleted edge of $G$ is dual to an edge of $T^*$ and merges the two adjacent complementary faces. Since $T^*$ connects all $F$ faces without cycles, after all $F-1$ deletions the complement has exactly one face. Thus $S\setminus G_f$ is an open disk.

The number of non-tree edges left in $G_f$ is
\[
 E-(V-1)-(F-1)=E-V-F+2.
\]
Euler's formula $V-E+F=2-2g$ gives
\[
 E-V-F+2=2g.
\]
Hence 
\begin{equation}
G_f = T \cup \{e_1,e_2,\ldots,e_{2g} \} 
\label{eq:spanning_tree}
\end{equation}
is a spanning tree together with $2g$ additional edges $(e_i)_i$, so $b_1(G_f)=2g$.  Since edges were only deleted, $L(G_f)\le L(G)$. Choose a root of $T$ and form a loop $a_i$ by connecting each edge $e_i$ with the root of $T$. The loops $(a_i)_{i=1,\ldots,2g}$ then generate the fundamental group $\pi_1(S)$. 
\end{proof}

\begin{remark} The primal spanning tree $T$ guarantees connectedness. The dual spanning tree $T^*$ tells us which primal edges to delete in order to collapse the collection of complementary faces to one face.  What remains outside the primal tree consists of exactly $2g$ edges, one for each generator of the fundamental group of the surface.
\end{remark}

\section{Half rank theorem for the mixed short edge graph}\label{sec:duality}

Our goal in this section is to prove that the mixed shorter-edge graph $\Gmin$ of the Voronoi graph and its dual contains at least half of the homology of the surface. The central topological statement is independent of the flat metric. Let $G$ and $G^*$ be dual cellular graphs on a closed oriented surface $S$ of genus $g\geq 1$. Let $A\subset E(G)$ be a set of edges and let $G_A$ be the subgraph of $G$ constructed with these edges. Let furthermore $G^*_{\overline A}$ be the subgraph of $G^*$ containing the dual edge $e^*$ of the edge pair $(e,e^*)$ precisely when $e \notin A$. To illustrate this construction, Figure~\ref{fig:G_A} displays a cellular graph $G$ on a translation surface of genus $2$, highlighting both the subgraph $G_A$ and its corresponding dual structure $G^*_{\overline{A}}$ for a specified subset $A \subset E(G)$. Set
\begin{equation}
 U_A=\im\bigl(H_1(G_A,\Z)\to H_1(S,\Z)\bigr),
 \qquad
 W_A=\im\bigl(H_1(G^*_{\overline A},\Z)\to H_1(S,\Z)\bigr).
\label{eq:UW}
\end{equation}

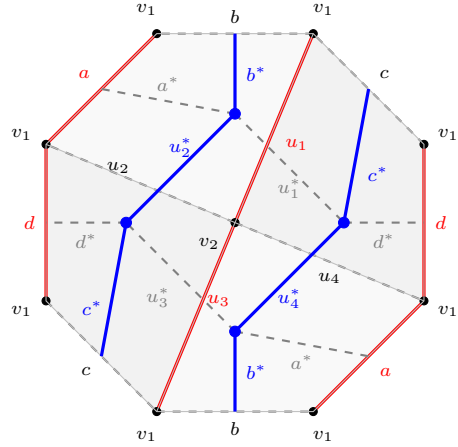
\begin{figure}[htbp]
\centering
\begin{tikzpicture}[scale=1]
    
    \def\R{2.706}       
    \def\r{2.5}         
    
    \foreach \i in {0,1,...,7} {
        \coordinate (P\i) at (22.5 + \i*45:\R);
        \coordinate (M\i) at (45 + \i*45:\r); 
    }
    
    \coordinate (Center) at (0,0);
    
    \coordinate (Ftop) at (0, 1.44);
    \coordinate (Fleft) at (-1.44, 0);
    \coordinate (Fbottom) at (0, -1.44);
    \coordinate (Fright) at (1.44, 0);

    \begin{scope}[xshift=0cm]
        \fill[gray!5] (Center) -- (P1) -- (P2) -- (P3) -- cycle;
        \fill[gray!10] (Center) -- (P3) -- (P4) -- (P5) -- cycle;
        \fill[gray!5] (Center) -- (P5) -- (P6) -- (P7) -- cycle;
        \fill[gray!10] (Center) -- (P7) -- (P0) -- (P1) -- cycle;

        \draw[very thick, red] (P2) -- (P3) node[midway, above left] {\tiny$a$};
        \draw[very thick, red] (P6) -- (P7) node[midway, below right] {\tiny$a$};
        \draw[very thick, red] (P7) -- (P0) node[midway, right] {\tiny$d$};
        \draw[very thick, red] (P3) -- (P4) node[midway, left] {\tiny$d$};
        \draw[very thick, red] (Center) -- (P1) node[midway, below right] {\tiny$u_1$};
        \draw[very thick, red] (Center) -- (P5) node[midway, above right] {\tiny$u_3$};

        \draw[thick, dashed, gray] (P1) -- (P2) node[midway, above, text=black] {\tiny$b$};
        \draw[thick, dashed, gray] (P5) -- (P6) node[midway, below, text=black] {\tiny$b$};
        \draw[thick, dashed, gray] (P0) -- (P1) node[midway, above right, text=black] {\tiny$c$};
        \draw[thick, dashed, gray] (P4) -- (P5) node[midway, below left, text=black] {\tiny$c$};
        \draw[thick, dashed, gray] (Center) -- (P3) node[midway, above left, text=black] {\tiny$u_2$};
        \draw[thick, dashed, gray] (Center) -- (P7) node[midway, below, text=black] {\tiny$u_4$};
        
        \foreach \i in {0,1,...,7} { \filldraw[black] (P\i) circle (1.5pt); }
        \filldraw[black] (Center) circle (1.5pt) node[below left=2pt] {\tiny$v_2$};

        \foreach \i in {0,1,...,7} { 
            \node at (22.5 + \i*45:\R + 0.35) [font=\small, black] {\tiny$v_1$};
        }
   
    \end{scope}
    
    \begin{scope}[xshift=6.0cm]
        
        \draw[thin, lightgray] (P0)--(P1)--(P2)--(P3)--(P4)--(P5)--(P6)--(P7)--cycle;
        \draw[thin, lightgray] (Center)--(P1); \draw[thin, lightgray] (Center)--(P3);
        \draw[thin, lightgray] (Center)--(P5); \draw[thin, lightgray] (Center)--(P7);

        \draw[very thick, blue] (Ftop) -- (M1) node[midway, right] {\tiny$b^*$};
        \draw[very thick, blue] (Fbottom) -- (M5) node[midway, right] {\tiny$b^*$};
        \draw[very thick, blue] (Fright) -- (M0) node[midway, below right] {\tiny$c^*$};
        \draw[very thick, blue] (Fleft) -- (M4) node[midway, below left] {\tiny$c^*$};
        \draw[very thick, blue] (Ftop) -- (Fleft) node[midway, above] {\tiny$u_2^*$};
        \draw[very thick, blue] (Fbottom) -- (Fright) node[midway, below ] {\tiny$u_4^*$};

        \draw[thick, dashed, gray] (Ftop) -- (M2) node[midway, above] {\tiny$a^*$};
        \draw[thick, dashed, gray] (Fbottom) -- (M6) node[midway, below] {\tiny$a^*$};
        \draw[thick, dashed, gray] (Fright) -- (M7) node[midway, below] {\tiny$d^*$};
        \draw[thick, dashed, gray] (Fleft) -- (M3) node[midway, below] {\tiny$d^*$};
        \draw[thick, dashed, gray] (Ftop) -- (Fright) node[midway, below ] {\tiny$u_1^*$};
        \draw[thick, dashed, gray] (Fbottom) -- (Fleft) node[midway, below left] {\tiny$u_3^*$};

        \filldraw[blue] (Ftop) circle (2pt);
        \filldraw[blue] (Fleft) circle (2pt);
        \filldraw[blue] (Fbottom) circle (2pt);
        \filldraw[blue] (Fright) circle (2pt);
 
    \end{scope}

\end{tikzpicture}
\caption{A cellular graph $G$ on a genus $2$ translation surface with $V(G)=\{v_1, v_2\}$, $E(G) = \{a, b, c, d, u_1, u_2, u_3, u_4\}$ and $4$ faces. The subgraph $G_A$ is formed by $A=\{a,d,u_1,u_3\}$ (solid red), while edges in $\overline{A}$ are dashed gray. $G^*_{\overline{A}}$ with $E(G^*_{\overline{A}}) =\{b^*, c^*, u_2^*, u_4^*\}$ (solid blue), linking the dual vertices residing in the faces of $G$. The remaining dual edges crossing $A$ are dashed gray.}
\label{fig:G_A}
\end{figure}

We will show that both $U_A$ and $W_A$ are primitive subgroups of $H_1(S,\Z)$ and that
\[
 \max\{\rank U_A,\rank W_A\}\geq g.
\]
To this end choose a small regular neighborhood $N(G_A)$ of $G_A$ in $S$ and set
\begin{equation}
 N_A=N(G_A),
 \qquad
 N^*_{\overline A}=\overline{S\setminus\operatorname{int}(N(G_A))}.
\label{eq:NAstar}
\end{equation}
Then $G_A$ is a deformation retract of $N_A$, and $G^*_{\overline A}$ is a deformation retract of $N^*_{\overline A}$. One way to see the second assertion is to start with $G_A=G$, when the complement of a regular neighborhood of $G$ retracts to the dual vertices, one in each face. Removing a primal edge $e$ and adding the corresponding dual edge $e^*$ replaces a band of the regular neighborhood on the primal side by the dual band on the complementary side. Repeating this operation gives the stated deformation retraction for arbitrary $A$. The required statement about the homology follows from the next proposition.

\begin{proposition}\label{prop:subsurface}
Let $S$ be a closed oriented surface of genus $g\geq1$, let $N\subset S$ be a compact subsurface, possibly disconnected, and set $M=\overline{S\setminus\operatorname{int}(N)}$. Define
\[
U=\im\bigl(H_1(N,\Z)\to H_1(S,\Z)\bigr),
\qquad
W=\im\bigl(H_1(M,\Z)\to H_1(S,\Z)\bigr).
\]
Then $U$ and $W$ are primitive subgroups of $H_1(S,\Z)$ and $\rank U+\rank W=2g$. In particular,
\[
\max\{\rank U,\rank W\}\geq g.
\]
\end{proposition}

\begin{proof}
We first describe $N$, $M$ and their boundaries with respect to their connected components:
\[
N=N_1\sqcup\cdots\sqcup N_n,
\qquad
M=M_1\sqcup\cdots\sqcup M_m,
\text{ \ and \ } \partial N=\partial M=C_1\sqcup\cdots\sqcup C_l.
\]
Let $g(N_i)$ and $g(M_j)$ be the genus of $N_i$ and $M_j$, respectively. We set 
\[
    g_N=\sum_{i=1}^n g(N_i) \quad \text{ \  and \ } g_M=\sum_{j=1}^m g(M_j).
\]
Our goal is to construct a basis of the homology $H_1(S,\Z)$ from simple closed curves in $M$ and $N$ including the boundary curves plus additional curves on the surface. We will then obtain our proposition by simply counting the curves from $N$ and $M$ we used in our construction. To this end, for every component $N_i$, choose a standard set of handle cycles
\[
\mathcal{A}_{N_i} := \{ a_{i,1},b_{i,1},\ldots,a_{i,g(N_i)},b_{i,g(N_i)} \} 
\]
in its interior, satisfying
\begin{equation}
\inter(a_{i,k}, b_{i,k})=1 \text{ \ and \ } \inter(b_{i,k}, a_{i,k})=-1   \text{ \ for all \ } k \in \{ 1,2,\ldots,g(N_i) \}.
\label{eq:inter_ab} 
\end{equation}

\begin{figure}[htbp]
    \centering
    \begin{tikzpicture}[scale=1.1, every node/.style={scale=0.9}]
        
        
        \draw[thick, fill=blue!10, draw=blue!80!black] 
            (-2, 1.5) -- (-4, 1.5) arc (90:270:1.5) -- (-2, -1.5) -- 
            (-2, -0.5) to[out=180, in=180, looseness=2] (-2, 0.5) -- cycle;
            
        \draw[thick, fill=green!10, draw=green!50!black] 
            (-2, 1.5) to[out=0, in=180] (2, 0.8) -- (2, -0.8) 
            to[out=180, in=0] (-2, -1.5) -- 
            (-2, -0.5) to[out=0, in=0, looseness=2] (-2, 0.5) -- cycle;

        \draw[thick, fill=orange!10, draw=orange!80!black] 
            (2, 0.8) -- (3.5, 0.8) arc (90:-90:0.8) -- (2, -0.8) -- cycle;

        
        \draw[red, very thick] (-2, 1.5) arc (90:270:0.2 and 0.5);
        \draw[red, very thick, dashed] (-2, 0.5) arc (-90:90:0.2 and 0.5);
        \node[red, above] at (-2, 1.6) {$C_1 = c_1$};

        \draw[red, very thick] (-2, -0.5) arc (90:270:0.2 and 0.5);
        \draw[red, very thick, dashed] (-2, -1.5) arc (-90:90:0.2 and 0.5);
        \node[red, below] at (-2, -1.6) {$C_2$};

        \draw[red, very thick] (2, 0.8) arc (90:270:0.2 and 0.8);
        \draw[red, very thick, dashed] (2, -0.8) arc (-90:90:0.2 and 0.8);
        \node[red, above] at (2, 0.9) {$C_3$};

        
        \begin{scope}[shift={(-4,0)}]
            \draw[thick] (-0.3, 0.1) to[out=-60, in=240] (0.3, 0.1);
            \draw[thick] (-0.2, -0.05) to[out=60, in=120] (0.2, -0.05);
            
            \draw[blue!80!black, thick, ->] (0, -0.05) arc (90:0:0.2 and 0.725);
            \draw[blue!80!black, thick] (0.2, -0.775) arc (0:-90:0.2 and 0.725);
            \draw[blue!80!black, thick, dashed] (0, -1.5) arc (270:90:0.2 and 0.725);
            \node[blue!80!black, right] at (0.2, -1.2) {$a_{1,1}$};
            
            \draw[blue!80!black, thick, ->] (1.1, 0) arc (0:90:1.1 and 0.8);
            \draw[blue!80!black, thick] (0, 0.8) arc (90:360:1.1 and 0.8);
            \node[blue!80!black, left] at (0.8, 1) {$b_{1,1}$};
            
            \fill[blue!80!black] (0.2, -0.786) circle (1.5pt);
        \end{scope}

        \begin{scope}[shift={(3.2,0)}]
            \draw[thick] (-0.3, 0.1) to[out=-60, in=240] (0.3, 0.1);
            \draw[thick] (-0.2, -0.05) to[out=60, in=120] (0.2, -0.05);
            
            \draw[orange!80!black, thick, ->] (0, -0.05) arc (90:0:0.15 and 0.375);
            \draw[orange!80!black, thick] (0.15, -0.425) arc (0:-90:0.15 and 0.375);
            \draw[orange!80!black, thick, dashed] (0, -0.8) arc (270:90:0.15 and 0.375);
            \node[orange!80!black, right] at (0.15, -0.6) {$a_{2,1}$};
            
            \draw[orange!80!black, thick, ->] (0.7, 0) arc (0:90:0.7 and 0.5);
            \draw[orange!80!black, thick] (0, 0.5) arc (90:360:0.7 and 0.5);
            \node[orange!80!black, right] at (0.3, 0.5) {$b_{2,1}$};
            
            \fill[orange!80!black] (0.15, -0.435) circle (1.5pt);
        \end{scope}

        \draw[purple, very thick, ->] (-1.2, 0) to[out=90, in=0] (-2, 1);
        \draw[purple, very thick] (-2, 1) to[out=180, in=90] (-2.8, 0) 
            to[out=-90, in=180] (-2, -1) to[out=0, in=-90] (-1.2, 0);
        \node[purple, right] at (-1.2, 0) {$d_1$};
        \fill[purple] (-2, 1) circle (1.5pt); 

        \node[blue!80!black] at (-4, 1.8) {\textbf{$N_1$ ($g_{N_1}=1$)}};
        \node[green!50!black] at (0.3, 0.3) {\textbf{$M_1$ ($g_{M_1}=0$)}};
        \node[orange!80!black] at (3.5, 1.2) {\textbf{$N_2$ ($g_{N_2}=1$)}};

    \end{tikzpicture}

    \vspace{0.2cm} 
    
    \begin{tikzpicture}[scale=1.5]
        \node[draw, circle, blue!80!black, fill=blue!10, thick, minimum size=0.8cm] (N1) at (-3,0) {$N_1$};
        \node[draw, circle, green!50!black, fill=green!10, thick, minimum size=0.8cm] (M1) at (0,0) {$M_1$};
        \node[draw, circle, orange!80!black, fill=orange!10, thick, minimum size=0.8cm] (N2) at (3,0) {$N_2$};
        
        \draw[very thick, red, dashed] (N1) to[out=35, in=145] node[midway, above=2pt] {$e_1 = C_1$ (non-tree)} (M1);
        \draw[very thick, blue] (N1) to[out=-35, in=-145] node[midway, below=2pt] {$C_2$ (tree)} (M1);
        \draw[very thick, blue] (M1) -- node[midway, above] {$C_3$ (tree)} (N2);
        
        \node[above] at (0, 1.2) {\textbf{Bipartite incidence graph $G_{NM}$}};
        \node[below] at (0, -1.2) {Spanning tree $T = \{C_2, C_3\}$, \ Betti number $b_1(G_{NM}) = 1$};
    \end{tikzpicture}
    
    \caption{A $g=3$ surface partitioned into three subsurfaces. The set of subsurfaces $N$ consists of $N_1$ and $N_2$ (both genus $1$), while the subsurface $M$ is a single genus $0$ pair of pants $M_1$. The basis of $H_1(S,\Z)$ requires the handle cycles $\{a_{1, 1},b_{1, 1},a_{2, 1},b_{2, 1}\}$ plus the non-tree cycle boundary $c_1=C_1$ and its dual loop $d_1$. Following the proposition: $\rank U = 5$ and $\rank W = 1$.}
    \label{fig:homology_basis_g3}
\end{figure}

Furthermore all other intersections among these handle cycles are equal to zero. Similarly, in every component $M_j$, choose standard handle cycles
\[
\mathcal{A}_{M_j}: = \{ \alpha_{j,1},\beta_{j,1},\ldots,
\alpha_{j,g(M_j)},\beta_{j,g(M_j)} \} 
\]
Since the $N$- and $M$-handle cycles lie in disjoint interiors, their mutual intersection numbers vanish. They are also disjoint from all the boundary curves $(C_j)_j$. We set 
\[
\mathcal{A}_{N} = \bigcup_{i=1}^n \mathcal{A}_{N_i} \text{ \ and \ } \mathcal{A}_{M} = \bigcup_{j=1}^m  \mathcal{A}_{M_j}. 
\]
We now construct a graph $G_{NM}$ that stores how the different parts of $N$ and $M$ are connected with each other. Let $G_{NM}$ be the bipartite incidence graph whose vertices are the components of $N$ and $M$ and whose edges correspond to the boundary curves $(C_i)_{i=1,\ldots,l}$. Here the edge associated with $C_j$ joins the component of $N$ and the component of $M$ having $C_j$ in their common boundary. Since $S$ is connected, $G_{NM}$ is connected. Choose a spanning tree $T\subset G_{NM}$ and let $(e_i)_i$ be the edges that are not in $T$. The number of these edges is equal to the first Betti number $b_1(G_{NM})$ of the graph $G_{NM}$
\begin{equation}
b:=  b_1(G_{NM}) =l-(n+m)+1  \quad \text{ \ and  \ }  \quad  G_{NM} = T \cup \{ e_1,\ldots,e_b \}.
\label{eq:b_1GNM}
\end{equation}
We now complete the handle cycles from $N$ and $M$ to a basis of the homology $H_1(S,\Z)$ of $S$ using cycles from the graph $G_{NM}$ and the corresponding boundary components of $N$ and $M$.  For every $e_i \notin T$, the graph $T \cup e_i$ contains a unique graph cycle $\tilde{d}_i$. Let $c_i$ denote the boundary curve corresponding to the non-tree edge $e_i$. We realize $\tilde{d}_i$ by a closed curve $d_i$ on $S$ in the following way. We choose $d_i$ such that it intersects $c_i$ exactly once and no other $c_j$. Furthermore $d_i$ should not intersect the handle cycles in $N$ and $M$ or any other $d_j$. This can always be achieved by rerouting it appropriately in the connected components. After choosing orientations,
\begin{equation}
\inter([c_i],[d_j])=  \delta_{ij},    \text{ \ for all \ } i,j \in \{1,2,\ldots,b \} 
\label{eq:inter_cd}
\end{equation}
In particular, neither $[c_i]$ nor $[d_i]$ is trivial in
$H_1(S,\Z)$. Let 
\begin{equation}
\mathcal{A} := \mathcal{A}_N \cup \mathcal{A}_M \cup \{ c_1,d_1,\ldots, c_b,d_b\} 
\label{eq:hom_basis_SGNM}
\end{equation}
There are exactly $2g$ curves in $\mathcal{A}$. This can be seen in the following way: Since $S$ is the union of $N$ and $M$ and as $N\cap M$ is a union of simple closed curves, we have for the Euler characteristic $\chi(S)$ of $S$:
\[
2-2g = \chi(S)=\chi(N)+\chi(M) = (2n-2g_N -l) +  (2m-2g_M -l).
\]
Since by \eqref{eq:b_1GNM} $l=b+n+m-1$, we have $g=g_N+g_M+b$ and $\mathcal{A}$ in \eqref{eq:hom_basis_SGNM} contains exactly $2g$ homology classes. Furthermore, it follows from \eqref{eq:inter_ab} and \eqref{eq:inter_cd} that up to reordering and reorienting of curves, the intersection matrix $J$ of the elements of $\mathcal{A}$ is the standard symplectic
matrix
\[
J=
\begin{pmatrix}
0&I_g\\
-I_g&0
\end{pmatrix}.
\]
Therefore the classes in $\mathcal{A}$ form a $\Z$-basis of $H_1(S,\Z)$. For instance, Figure~\ref{fig:homology_basis_g3} depicts the construction of the bipartite incidence graph $G_{NM}$ alongside a homology basis for a surface of genus $3$. Moreover, $U$ is generated by $\mathcal{A}_N \cup  \{ c_1,\ldots, c_b\}$ and $W$ is generated by $\mathcal{A}_M \cup  \{ c_1,\ldots, c_b\}$ and both $U$ and $W$ are primitive subgroups. It follows that
\[
\rank U=2g_N+b,
\quad \text{ \ and \ } \quad
\rank W=2g_M+b, \quad \text{ \ so \ } \rank U + \rank W = 2g.
\]
Therefore at least one of $U$ and $W$ has rank at least $g$. This proves our proposition. 
\end{proof}

We now return to the mixed graph. To construct $\Gmin$, we choose exactly one edge from each dual pair $(e,e^*)$. If $A$ is the set of chosen primal edges of $G_V$, then
\[
\Gmin=G_A\cup G^*_{\overline A}.
\]
The graph may be disconnected, but its homology image contains both $U_A$ and $W_A$ (see \eqref{eq:UW}).

\begin{theorem}[Mixed-graph half-rank theorem]\label{thm:halfrank}
Let $G$ be a cellular graph on a closed oriented surface $S$ of genus $g\geq1$, and let $G^*$ be a dual graph of $G$. Every mixed graph $G_{mix}$ obtained by choosing exactly one edge from each dual edge pair satisfies
\[
\rank \left( \im\bigl(H_1(G_{mix},\Z)\to H_1(S,\Z)\bigr)\right) \geq g.
\]
Moreover $\im\bigl(H_1(G_{mix},\Z)\to H_1(S,\Z)\bigr)$ is a primitive subgroup of $H_1(S,\Z)$.
\end{theorem}

We obtain Lemma \ref{lem:Gmin_intro} in the introduction from this theorem and Proposition \ref{prop:mixedlength}, 2.).

\begin{proof}
Setting
\[
N=N(G_A), \qquad  M=\overline{S\setminus\operatorname{int}(N)},
\qquad U=U_A, \qquad  W=W_A,
\]
Proposition \ref{prop:subsurface} and the deformation retractions above imply
\[
\max\{\rank U_A,\rank W_A\}\geq g.
\]
Since both $U_A$ and $W_A$ are contained in the homology image of $G_{mix}$, the rank assertion follows. For primitivity, let $N_{mix} = N(G_{mix})$ be a regular neighborhood of the entire embedded graph $G_{mix}$. Since $N_{mix}$ deformation retracts onto $G_{mix}$, Proposition \ref{prop:subsurface}, applied to $N_{mix}$, shows directly that the image of $H_1(G_{mix},\Z)$ in $H_1(S,\Z)$ is primitive.
\end{proof}

\section{Systolic estimates for metric graphs}\label{sec:graph}

We now provide the graph-theoretic ingredients for the length estimates. We recall that the \textit{systole} $\sys(G)$, or \textit{girth}, of a metric graph is the length of a shortest noncontractible simple cycle. The following result is due to Bollobas and Szemeredi \cite{bsz}, improving an earlier estimate of Bollobas and Thomason \cite{bt}.

\begin{theorem}[Bollobas-Szemeredi-Thomason]\label{thm:bst}
Let $G$ be a connected metric graph with first Betti number $b\geq1$ and total edge length $L(G)=L$. Then
\[
\sys(G)\leq 4\frac{L}{b}\log(b+1).
\]
\end{theorem}

This theorem follows from the girth estimate of Bollobas-Szemeredi \cite{bsz} by subdivision of the edges. The simplified statement above for metric graphs can be deduced directly from \cite{sa}, Theorem 2.2 and (2.3). Before applying it, we isolate the integral homology core that will be used below.

\begin{lemma}[Primitive homology core]\label{lem:core}
Let $G$ be a finite metric graph, possibly disconnected, embedded in a closed oriented surface $S$, and set
\[
r=\rank \im
\bigl(H_1(G,\Z)\longrightarrow H_1(S,Z)\bigr)\geq 1.
\]
Then there exists a subgraph $G_c\subseteq G$ such that $L(G_c)\leq L(G)$ and  $b_1(G_c)=r$ and the map
\[
 H_1(G_c,\Z)\longrightarrow H_1(S,\Z)
\]
induced by the inclusion $G_c \hookrightarrow S$ is injective and has primitive image. More precisely, $G_c$ may be chosen as the union of $r$ simple fundamental graph cycles whose images can be extended to a basis of $H_1(S,\Z)$.
\end{lemma}

We note here that Proposition \ref{prop:subsurface} implies that, as $G$ is embedded, the image of $H_1(G,\Z)$ in $H_1(S,\Z)$ is primitive. 

\begin{proof}
Choose a spanning forest $T \subset G$. For every edge $e\in E(G)\setminus E(T)$, let $p_e\subseteq T$ be the unique simple path joining the endpoints of $e$ and set
\[
a_e=e\cup p_e.
\]
The classes $([a_e])_e$, for $e\in E(G)\setminus E(T)$, form the usual
fundamental cycle basis of $H_1(G,\Z)$. Let
\[
\Phi:H_1(G,\Z)\longrightarrow H_1(S,\Z)
\]
denote the homomorphism induced by the inclusion $G\hookrightarrow S$. Since its image has rank $r$, we can choose distinct non-tree edges $e_1,\ldots,e_r$ such that $(\Phi([a_{e_i}])_{i=1,\ldots,_r}$ are $\Z$-linearly independent. Set
\[
G_c=\bigcup_{i=1}^r a_{e_i}.
\]
Since $G_c\subseteq G$, we have $L(G_c)\leq L(G)$. The tree edges occurring in $G_c$ form a forest, and each of the edges $e_i$ adds exactly one independent cycle. Consequently, $b_1(G_c)=r$, and the classes $([a_{e_i}])_{i=1,\ldots,r}$ form a $\Z$-basis of $H_1(G_c,\Z)$. The images in $H_1(S,\Z)$ are $\Z$-linearly independent by construction. Therefore
\[
H_1(G_c,\Z)\longrightarrow H_1(S,\Z)
\]
is injective. Finally, let $N(G_c)$ be a regular neighborhood of $G_c$. Since $N(G_c)$ deformation retracts onto $G_c$, Proposition \ref{prop:subsurface} applied to $N(G_c)$ implies that
\[
\im
\bigl(H_1(G_c,\Z)\longrightarrow H_1(S,\Z)\bigr)
\]
is a primitive subgroup of $H_1(S,\Z)$. 
\end{proof}

We can apply Theorem \ref{thm:bst} iteratively. Together with Lemma \ref{lem:core} this gives short cycles that can be extended to a homology basis.

\begin{lemma}[Length bound on a primitive cycle basis]\label{lem:length_est_log}
Let $G$ be a finite metric graph, possibly disconnected, embedded in a closed oriented surface $S$. Suppose that
\[
H_1(G,\Z)\longrightarrow H_1(S,\Z)
\]
is injective. Let $b=b_1(G)\geq 1$ be the first Betti number of $G$ and $L=L(G)$ be the total length of edges of $G$. Then $G$ contains $b_1(G)$ simple graph cycles $(a_i)_{i=1,\ldots,b}$ that can be extended to a homology basis of $H_1(S,\Z)$ and which satisfy
\[
\ell(a_i)
\leq
\frac{4L}{b-i+1}\log(b-i+2),
\qquad \text{ \ for all \ } i \in \{1,2,\ldots,b \}.
\]
\end{lemma}

\begin{proof} We will construct inductively a sequence of subgraphs
\[
G=G_1\supset G_2\supset\cdots\supset G_{b+1},
\]
together with simple cycles $a_i\subset G_i$ satisfying the length bound of the lemma and edges $e_i\subset a_i$. We obtain $G_{i+1}$ by deleting the interior of $e_i$ from $G_i$. Deleting $e_i$ will decrease the first Betti number of the graph by one. Suppose $G_i$ has constructed. Let $\widetilde L_i\leq L(G_i)$ be the combined length of the connected components $(C_j)_j$ of $G_i$ with positive first Betti number. At least one such component $C_i$ satisfies
\[
\frac{L(C_i)}{b_1(C_i)}
\leq
\frac{\widetilde L_i}{b_1(G_i)} \leq \frac{ L(G_i)}{b_1(G_i)}.
\]
Indeed, otherwise summing the strict opposite inequalities over all components with positive first Betti number gives a contradiction in the first inequality. Applying Theorem \ref{thm:bst} to $C_i$, choose a simple cycle $a_i$ with
\[
\ell(a_i)
\leq
4\frac{L(C_i)}{b_1(C_i)}\log\bigl(b_1(C_i)+1\bigr)
\leq 4\frac{L}{b-i+1}\log(b-i+2).
\]
Choose an edge $e_i$ of $a_i$ and delete its interior. Since $e_i$ lies on a cycle, this does not change the number of connected components and lowers the total first Betti number by exactly one. Continue until $a_1,\ldots,a_b$ have been selected. After the $b$ deletions, the remaining graph $G_{b+1}$ is a spanning forest of $G$, and
\[
E(G)\setminus E(G_{b+1})=\{e_1,\ldots,e_b\}.
\]
Taking $G_{b+1}$ as the spanning forest, let $a_{e_i}$ be the fundamental cycle determined by the edge $e_i$. Then $\mathcal{A}_G=\{a_{e_1},\ldots,a_{e_b}\} $ is a basis of $H_1(G,\mathbb Z)$. For each $i$, $a_i\subset G_i$, hence $a_i$ contains none of $e_1,\ldots,e_{i-1}$. Since $a_{e_j}$ is the unique element of $\mathcal{A}_G$ containing the non-tree edge $e_j$, the coordinate matrix of $[a_1],\ldots,[a_b]$ with respect to $\mathcal{A}_G$ is upper triangular with diagonal entries $\pm1$. It is therefore unimodular, and the $([a_i])_i$ also form a $\Z$-basis of $H_1(G,\Z)$. Since the homology map is injective, their images form a $\Z$-basis of the image $\im \bigl(H_1(G,\Z)\longrightarrow H_1(S,\Z)\bigr)$. Since $G$ is embedded in $S$, Proposition \ref{prop:subsurface} applied to a regular neighborhood of $G$ shows that this image is a primitive subgroup. Hence the classes $([a_i])_{i=1,\dots,b}$ form a basis of a primitive subgroup of $H_1(S,\Z)$ and can be extended to a basis of $H_1(S,\Z)$.
\end{proof}

We now apply Lemmas \ref{lem:core} and \ref{lem:length_est_log} to the mixed shorter-edge graph $\Gmin$.

\begin{theorem}[Half-rank log bound]\label{thm:mixedmain}
Let $S$ be a closed translation surface of genus $g\geq2$ with $\Area(S)=4\pi g$. Let $V$ be the Voronoi graph on $S$ whose seeds are the cone points, let $N$ be the number of edges of $V$, and let $\Gmin$ be the mixed shorter-edge graph. Then $\Gmin$ contains $g$ simple graph cycles $(\alpha_i)_{i=1,\ldots,g}$ whose homology classes generate a primitive subgroup of rank $g$ in $H_1(S,\Z)$, and hence can be extended to a $\Z$-basis of $H_1(S,\Z)$, such that
\[
\ell(\alpha_i)
\leq
\frac{4\sqrt{2N\Area(S)}}{g-i+1}\log(g-i+2)
\leq
\frac{70g}{g-i+1}\log(g-i+2),
\qquad \text{ \ for all \ } i \in \{1,2,\ldots,g \}.
\]
\end{theorem}

This is a slightly more detailed version of Theorem \ref{thm:half_rank_intro} in the introduction.

\begin{proof}
Set
\[
U=\im\bigl(H_1(\Gmin,\Z)\to H_1(S,\Z)\bigr),
\qquad
r=\rank U.
\]
By Theorem \ref{thm:halfrank}, $g\leq r\leq2g$. Lemma \ref{lem:core} gives a subgraph $G_c\subseteq\Gmin$ with $b_1(G_c)=r, L(G_c)\leq L(\Gmin)$ and the inclusion map $H_1(G_c,\Z)\longrightarrow H_1(S,\Z)$ is injective and has primitive image. Set 
\[
 U_c = \im(H_1(G_c,\Z)\longrightarrow H_1(S,\Z))
\]
Then $U_c$ is a primitive subgroup of rank $r$. We can therefore apply Lemma \ref{lem:length_est_log} to $G_c$. We obtain a $\Z$-basis $\alpha_1,\ldots,\alpha_r$ of $U_c$ satisfying
\[
\ell(\alpha_i)
\leq
\frac{4L(\Gmin)}{r-i+1}\log(r-i+2).
\]
The function $\log(x+1)/x$ is decreasing for $x\geq1$. Since $r\geq g$, for $1\leq i\leq g$ we therefore have
\[
\frac{\log(r-i+2)}{r-i+1}
\leq
\frac{\log(g-i+2)}{g-i+1}.
\]
Thus the first $g$ basis elements satisfy
\[
\ell(\alpha_i)
\leq
\frac{4L(\Gmin)}{g-i+1}\log(g-i+2).
\]
Because $U_c$ is primitive, the $([\alpha_i])_{i=1,\ldots,g}$ generate a primitive rank-$g$ subgroup of $H_1(S,\Z)$ and can be extended to a homology basis. Finally, Proposition \ref{prop:mixedlength} and Proposition \ref{prop:edgecount} give
\[
L(\Gmin)\leq\sqrt{8\pi gN}
\quad \text{ \ and \ } \quad
N\leq12(g-1),
\]
which yields the stated upper bound.
\end{proof}

Expressing this result in terms of a portion of the homology, we have:

\begin{corollary}\label{cor:proportion}
Let $S$ be a translation surface of genus $g\geq2$ with $\Area(S)=4\pi g$. Then for any $\lambda \in (0,1)$ there are $\lfloor\lambda g\rfloor$ homologically independent simple closed curves $(\alpha_i)_{i=1,\ldots,\lfloor\lambda g\rfloor}$ whose classes can be extended to a $\Z$-basis of $H_1(S,\Z)$ and which satisfy
\[
\ell(\alpha_i)
\leq
\frac{70}{1-\lambda}\log(g+1),
\quad \text{ \ for all \ }  i \in \{1,2,\ldots,\lfloor\lambda g\rfloor \}.
\]
\end{corollary}

Thus we obtain a uniform $O(\log g)$ bound for a primitive subgroup of rank arbitrarily close to one half of the homology. The cut-graph construction gives a complementary result for the full rank with a different geometric dependence.

\begin{theorem}[Short integral homology basis]\label{thm:short_basis}
Let $S$ be a closed translation surface of genus $g\geq2$ with $\Area(S)=4\pi g$, and let $\delta$ be a shortest saddle connection on $S$. Then there exist $2g$ simple closed curves $(\alpha_i)_{i=1,\ldots,2g}$ whose homology classes form a $\Z$-basis of $H_1(S,\Z)$ and such that
\[
\ell(\alpha_i)
\leq
\frac{32\pi g}{(2g-i+1)\,\ell(\delta)}\log(2g-i+2),
 \quad \text{ \ for all \ }  i \in \{1,2,\ldots,2g \}.
\]
\end{theorem}

Together with Proposition \ref{prop:mixedlength}, this proves Theorem \ref{thm:full_rank_intro} from the introduction.

\begin{proof}
By Lemma \ref{lem:cut graph}, the cut graph $G_f$ satisfies $b_1(G_f)=2g$ and $\pi_1(G_f)\twoheadrightarrow\pi_1(S)$. Hence
\[
H_1(G_f,\Z)\xrightarrow{\;\cong\;}H_1(S,\Z).
\]
Therefore we can apply Lemma \ref{lem:length_est_log} to $G_f$ with $b=2g$ using Proposition \ref{prop:mixedlength} for the length estimates. This proves our claim.
\end{proof}

\begin{remark}[The stratum $\mathcal{H}(2g-2)$]
If $S\in\cH(2g-2)$, then the Voronoi decomposition has one face and the full Voronoi graph is already a cut graph. Thus one may take $G_f=G_V$ and apply Lemma \ref{lem:length_est_log} directly. 
\end{remark}

\section*{Acknowledgments and disclosure}

This manuscript was developed from preliminary notes and subsequent mathematical discussion between the authors and OpenAI's ChatGPT. The system assisted with synthesis, organization, and editorial drafting. The authors take complete responsibility for the mathematical claims, the references, and any version submitted for circulation or publication.

\vspace{0.4cm}

\noindent Peter Buser \\
\noindent Department of Mathematics, Ecole Polytechnique F\'ed\'erale de Lausanne\\
\noindent Station 8, 1015 Lausanne, Switzerland\\
\noindent e-mail: \textit{peter.buser@epfl.ch}\\
\\
\noindent Achintya Dey \\
\noindent Department of Mathematics, Indian Institute of Technology Kanpur\\
\noindent Kanpur 208016, Uttar Pradesh, India\\
\noindent e-mail: \textit{achintd@iitk.ac.in}\\
\\
\noindent Eran Makover\\
\noindent Department of Mathematics, Central Connecticut State University\\
\noindent 1615 Stanley Street, New Britain, CT 06050, USA\\
\noindent e-mail: \textit{makovere@ccsu.edu}\\
\\
\noindent Bjoern Muetzel \\
\noindent Department of Mathematics, Eckerd College \\
\noindent 4200 54th avenue South, St. Petersburg, FL 33711, USA\\
\noindent e-mail: \textit{bjorn.mutzel@gmail.com}\\ 

\end{document}